\documentclass[12pt]{article}
\usepackage[amsmath]{e-jc}
\usepackage{enumitem}

\newtheorem{thm}{\bf Theorem}[section]
\newtheorem{lem}[thm]{\bf Lemma}

\newtheorem{prop}[thm]{\bf Proposition}

\theoremstyle{definition}
\newtheorem{defn}[thm]{Definition}

\newtheorem{rem}[thm]{Remark}
\numberwithin{equation}{section}
\def\C{\mathbb{C}}
\def\N{\mathbb{N}}
\def\R{\mathbb{R}}
\def\Z{\mathbb{Z}}
\def\S{\mathfrak{S}}
\def\B{\mathcal{B}}
\def\I{\mathcal{I}}
\def\pos{\mathrm{pos}}

\newcommand{\des}{{\rm des}}
\newcommand{\maj}{{\rm maj}}

\newcommand{\fmaj}{{\rm fmaj}}

\newcommand{\Neg}{{\rm neg}}

\def\Des{\mathop{\rm Des}\nolimits}

\title{Some identities involving $q$-Stirling numbers\\ 
of the second kind   in type B}
\author{Ming-Jian Ding\authornote{1}
\and
Jiang Zeng\authornote{2}
}
\authortext{1}{School of Mathematic Sciences,
 Dalian University of Technology, Dalian 116024, P. R. China (\email{ding-mj@hotmail.com}).}
\authortext{2}{Universite Claude Bernard Lyon 1, ICJ UMR5208, CNRS, Centrale Lyon, INSA Lyon, Université Jean Monnet, 69622, Villeurbanne Cedex, France 
(\email{zeng@math.univ-lyon1.fr}).}
\dateline{June 19, 2023}{Jan 4, 2024}{TBD}
\Copyright{The authors. Released under the CC BY license (International 4.0).}
\begin{document}
\maketitle
\begin{abstract}
The recent interest in type B 
$q$-Stirling numbers of the second kind  prompted us to give a type  B analogue of a classical identity connecting the $q$-Stirling numbers of the second kind and Carlitz's major $q$-Eulerian numbers, which turns out to be a $q$-analogue of an identity due to Bagno, Biagioli and Garber. We provide a combinatorial proof  of this identity and an algebraic  proof of  a more general identity for colored permutations.
In addition, we  prove some $q$-identities about the $q$-Stirling numbers of the second kind in types A, B and D.
\end{abstract}

\section{Introduction}
The Stirling number of the second kind, denoted $S(n,k)$,
is the number of ways to partition $n$ distinct objects into $k$ nonempty subsets.
It satisfies the well-known triangular recurrence
\begin{equation*}
  S(n,k)=S(n-1,k-1)+kS(n-1,k)
\end{equation*}
with the initial conditions $S(0,k)=\delta_{0k}$, where $\delta_{ij}$ is the Kronecker delta.
Carlitz \cite{Car48} introduced the type A $q$-Stirling numbers of the second kind $S[n,k]$ by
\begin{equation}\label{rec+stirling+q+A}
  S[n,k] := S[n-1,k-1]+[k]_q\,S[n-1,k],
\end{equation}
where $[k]_q := 1 + q + q^2 +\dots+ q^{k-1}$
for $k\geq 1$ and $[0]_q := 0$, and  $S[0,k]=\delta_{0k}$.

Let $\S_n$ be the symmetric group on the set $[n]=\{1,2,\ldots,n\}$.
An element $\pi \in \S_n$ is written as $\pi=\pi_1\pi_2\cdots \pi_n$.
The descent set of $\pi \in \S_n$ is defined by
\begin{equation*}
\Des(\pi):=\{ i \in [n-1]~|~ \pi_i > \pi_{i+1} \}
\end{equation*}
and the cardinality of $\Des(\pi)$ is called
the number of descents of $\pi$, denoted  $\des(\pi)$.
The Eulerian number $A_{n,k}$ is  the number of $\pi \in \S_n$ with $k$ descents.
There exists a well-known identity connecting the Stirling numbers of the second kind and Eulerian numbers as follows:
\begin{equation}\label{eq+eulerian+stir+A}
  k!S(n,k)=\sum_{\ell=1}^kA_{n,\ell-1}\binom{n-\ell}{k-\ell}
\end{equation}
for all nonnegative integers $0 \leq k \leq n$.
A combinatorial proof of identity~\eqref{eq+eulerian+stir+A} in terms of the ordered set partitions and permutations is quite easy and well known, see \cite[Theorem 1.17]{Bo04}, for example. 

The $q$-binomial coefficients are defined  for $n, k \in \N$ by
\begin{equation*}
   {n \brack k}_{q} := \frac{[n]_q!}{[k]_q![n-k]_q!}\quad\text{for}\quad 0 \leq k \leq n,
\end{equation*}
 where $[n]_q! := [1]_q[2]_q \cdots[n]_q$ is the $q$-factorial of $n$.
To give a $q$-analogue of identity \eqref{eq+eulerian+stir+A}  we need to find a suitable \emph{Mahonian statistic} 
 over permutations, that is, a statistic  whose generating function 
  over $\S_n$ is $[n]_q!$. It turns out that MacMahon's 
 \emph{major index}~\cite{Mac60} is a good fit for our $q$-analogue. Recall that  
 the  \emph{major index} ($\maj$)  of  $\pi \in \S_n$  is defined by
\begin{equation*}
  \maj(\pi) := \sum_{i \in \Des(\pi)}i.
\end{equation*}
We define the corresponding  $q$-analogue of Eulerian polynomial (of type A) by
\begin{equation}\label{def+q+eulerian+poly+A}
  A_n(t,q) := \sum_{\pi \in \S_n}t^{\des(\pi)}q^{\maj(\pi)}
            = \sum_{k=0}^nA_{n,k}(q)t^k.
\end{equation}
The reader is referred to \cite{Fo10, peter15} and 
 references therein  for further  $q$-Eulerian polynomials. 

Using analytic method, Zeng and Zhang \cite[Proposition 4.5]{ZZ94} proved
the following  $q$-analogue of identity~\eqref{eq+eulerian+stir+A}~\footnote{Proposition 4.5 in \cite{ZZ94} is actually a \emph{fractional version} of~\eqref{eq+eulerian+stir+A+q} and valid for $n\in \C$.}
\begin{equation}\label{eq+eulerian+stir+A+q}
    q^{\binom{k}{2}}[k]_q!S[n,k]
  = \sum_{\ell=1}^kq^{k(k-\ell)}A_{n,\ell-1}(q){n-\ell \brack k-\ell}_{q}
\end{equation}
for nonnegative integers
$0\leq k \leq n$. In 1997, in order  to  give a combinatorial proof of~\eqref{eq+eulerian+stir+A+q}, 
Steingr\'{\i}msson \cite{Stein20} proposed several statistics on ordered set partitions and conjectured that 
their generating functions  were given by either side of~\eqref{eq+eulerian+stir+A+q}.  In the following years,
 Zeng et al.~\cite{KZ02,IKZ08,KZ09} confirmed  all his conjectures, and finally 
Remmel and Wilson~\cite[Section 5.1]{RW15} found
 a combinatorial proof of~\eqref{eq+eulerian+stir+A+q}
using  the major index on the starred permutations.

This paper arose from the desire to give a type B analogue of \eqref{eq+eulerian+stir+A+q}.
In analogy with the usual (type A) Stirling numbers of the second kind (see  \cite{Za81,DL94,BBG19,SS22}),
the \emph{type B  Stirling numbers of the second kind} $S_B(n,k)$  can be defined by
\begin{equation*}
  S_B(n,k) := S_B(n-1,k-1)+ (2k+1)S_B(n-1,k)
\end{equation*}
with the initial conditions $S_B(0,k)=\delta_{0k}$.
%We refer the reader to 
%for the geometric and combinatorial interpretations of $S_B(n,k)$.

For integer $i \in \Z$ we denote
its opposite integer  $-i$ by $\overline{i}$.
Let $\B_n$ be the group of signed permutations of $[n]$,
i.e., the set of all permutations on the set $[\pm n]:=\{\overline{n},\ldots,\overline{1},1,\ldots,n\}$
such that $\pi(\overline{i})=\overline{\pi(i)}$.
In what follows, we write $\pi(i)$ as $\pi_i$ for
$i\in [\pm n]$ and use the \emph{natural order}
on $\langle n\rangle:=\{\overline{n},\ldots,\overline{1},0,1,\ldots,n\}$, namely,
\begin{equation*}
  \overline{n} < \cdots < \overline{1} < 0 < 1 < \cdots < n.
\end{equation*}
The type B descent set of $\pi \in \B_n$ \cite[Section 11.5.2]{peter15} is defined by
\begin{equation*}
\Des_B(\pi) = \{i \in \{0\}\cup[n-1] ~|~ \pi_i > \pi_{i+1}\},
\end{equation*}
with $\pi_0=0$, and the cardinality of  $\Des_B(\pi)$ is called
the number of type B descents of $\pi$, denoted $\des_B(\pi)$.

Let  $B_{n,k}$ be  the number of permutations in $\B_n$ with $k$ descents.
By a bijection between the set of ordered type B set partitions and the set of signed permutations with separators, Bagno, Biagioli and Garber~\cite{BBG19} combinatorially proved the following type B analogue of~\eqref{eq+eulerian+stir+A}:
\begin{equation}\label{eq+eulerian+stir+B}
  2^kk!\,S_B(n,k)=\sum_{\ell=0}^kB_{n,\ell}\binom{n-\ell}{k-\ell}
\end{equation}
for all nonnegative integers $0 \leq k \leq n$.

Recently Sagan and Swanson~\cite{SS22} studied
the \emph{type B $q$-Stirling numbers of the second kind $S_B[n,k]$},
which are defined by the recurrence relation
\begin{equation}\label{def+stirling+B+q}
  S_B[n,k] := S_B[n-1,k-1]+[2k+1]_q\,S_B[n-1,k]
\end{equation}
with the initial conditions $S_B[0,k]=\delta_{0k}$,
see \cite[Section 1.10]{SW21} and \cite{BGK22} for related works.

\begin{rem}
Chow-Gessel~\cite[Eq. (18) and Proposition 4.2]{CG07} defined  a kind of type B $q$-Stirling numbers of the second kind $S_{n,k}(q)$ by the following recurrence relation
\begin{equation*}
S_{n,k}(q):=q^{2k-1}(1+q)S_{n-1,k-1}(q)+[2k+1]_qS_{n-1,k}(q)
\end{equation*}
with the initial conditions $S_{n,0}(q)=1$ for $n \geq 0$.
It is routine to verify that the above two 
types B $q$-Stirling numbers of the second kind are related as follows
\begin{equation}
S_{n,k}(q)=(1+q)^{k} q^{k^2}S_B[n,k].
\end{equation}
\end{rem}

Adin and Roichman~\cite{AR01} defined the \emph{flag-major index} of $\pi \in \B_n$ as follows
\begin{equation}\label{def:fmaj}
  \fmaj(\pi) := \sum_{i \in \Des_B(\pi)}2i+\Neg(\pi),
\end{equation}
where $\Neg(\pi)$ is the number of negative elements in $\pi$, i.e., $|\{i\in [n]: \pi_i < 0\}|$.
Then, as a $q$-analogue of Eulerian polynomial of type B,
Chow and Gessel~\cite{CG07} studied the enumerative polynomials of statistic $(\des_B, \fmaj)$ over $\B_n$,
\begin{equation}\label{def+q+eulerian+poly+B}
  B_n(t,q) := \sum_{\pi \in B_n}t^{\des_B(\pi)}q^{\fmaj(\pi)}= \sum_{k=0}^nB_{n,k}(q)t^k.
\end{equation}

In this paper, using Sagan and Swanson's $q$-Stirling numbers of the second kind in type B \cite{SS22}
and Chow and Gessel's $q$-Eulerian numbers of type B,
we prove  a $q$-analogue of Bagno et al.'s identity~\eqref{eq+eulerian+stir+B}. The following is our first main  result.
% is a $q$-analogue of Bagno, Biagioli and Garber's identity~\eqref{eq+eulerian+stir+B}.

\begin{thm}\label{thm+stirling+frobenius+B}
For $0 \leq k \leq n$ we have
\begin{equation}\label{eq+eulerian+stir+B+q+trans}
    [2]^k[k]_{q^2}!S_B[n,k]
  = \sum_{\ell=0}^{k}q^{k(k-2\ell)}B_{n,\ell}(q)
    {n-\ell \brack k-\ell}_{q^2}.
\end{equation}

\end{thm}

We shall provide a combinatorial proof for Theorem \ref{thm+stirling+frobenius+B} in Section \ref{section+combinatorial+proof}.
In Section \ref{section+generalization},
we define a $q$-Stirling numbers of the second kind in type D and give $q$-analogues of
some known identities connecting the Stirling numbers of the second kind in types A, B and D.
Next, we prove algebraically a general identity (see Theorem \ref{thm+stirling+frobenius+r})
between the $r$-colored $q$-Stirling numbers of the second kind and $q$-Eulerian numbers of colored permutations
in Section \ref{section+algebraic+proof}.
Note that the proof of Theorem \ref{thm+stirling+frobenius+r} yields another  proof of Theorem \ref{thm+stirling+frobenius+B}.

\section{Combinatorial proof of Theorem~\ref{thm+stirling+frobenius+B}}\label{section+combinatorial+proof}

In this section, we give a combinatorial proof of \eqref{eq+eulerian+stir+B+q+trans}
by generalizing Remmel and Wilson's proof of identity~\eqref{eq+eulerian+stir+A+q} in \cite{RW15}.
Our strategy is to study the polynomial
\begin{equation}
  \sum_{\pi \in \B_n}q^{\fmaj(\pi)}\prod\limits_{i=1}^{\des_B(\pi)}\left(1+\frac{z}{q^{2i-1}} \right)
\end{equation}
in $\R[q][z]$ and interpret the coefficient of $z^k$ combinatorially in two different ways.

\subsection{Permutations of type B}
For any $\pi=\pi_1\pi_2\cdots\pi_n \in \B_n$,
we say that an index $i\in [n-1]$ has $\pi$-\emph{sign type} $++$ (resp., $--$, $+-$, $-+$)
if the sign of $\pi_i$ is positive (resp., negative, positive, negative) and that of
$\pi_{i+1}$ is positive (resp., negative, negative, positive).

In the rest of this section, we denote by $\Pi_1$ (resp., $\Pi_2$, $\Pi_3$)
the set of descents of $\pi$ with $\pi$-sign type $++$ (resp., $--$, $+-$)
and by $\Pi'_1$ (resp., $\Pi'_2$, $\Pi'_3$) the set of ascents of $\pi$ with $\pi$-sign type $++$ (resp., $--$, $-+$).

For any $\pi \in \B_n$, define the mapping $\psi:\pi\to \widetilde{\pi} $ on $\B_n$ by
\begin{equation*}
     \widetilde{\pi}_i
 := \left\{
   \begin{array}{lll}
      \pi_{n+1-i}-n-1,  &  \text{if} ~ \pi_{n+1-i} > 0; \\
      \pi_{n+1-i}+n+1,  &  \text{if} ~ \pi_{n+1-i} < 0. \\
   \end{array}
   \right.
\end{equation*}
For example, if $\pi=1\,5\,\overline{3}\,4\,6\,\overline{2}$, then $\widetilde{\pi}=5\,\overline{1}\,\overline{3}\,4\,\overline{2}\,\overline{6}$.

\begin{rem}
Let $r: \pi\mapsto \pi^r$ be the \emph{reversing operator} on $\B_n$
defined by $\pi^r_i=\pi_{n+1-i}$  and
$c: \pi \mapsto \pi^c$  the type B \emph{completion operator} on $\B_n$ defined by $\pi^c_i=\varepsilon_i\cdot (n+1-|\pi_i|)$, where
$\varepsilon_i=1$ if $\pi_i<0$ and $-1$ if
$\pi_i>0$ for $i\in [n]$. It is easy to verify that
$\widetilde{\pi}=(\pi^r)^c$.
\end{rem}

Clearly, if $i$ is a descent (resp., an ascent) position in $\pi \in \B_n$ and the product of $\pi_i$ and $\pi_{i+1}$ is positive,
then $n-i$ is an ascent (resp., a descent) position in $\widetilde{\pi}$;
if $i$ is a descent (resp., an ascent) position in $\pi \in \B_n$ and the product of $\pi_i$ and $\pi_{i+1}$ is negative,
then $n-i$ is a descent (resp., an ascent) position in $\widetilde{\pi}$.

In fact, the mapping $\psi$ is a bijection between all permutations in $\B_n$ with $k$ descents
and all permutations in $\B_n$ with $n-k$ descents by the following result.

\begin{lem}\label{prop+bijection+B}
The mapping $\psi$ is a bijection on $\B_n$ such that
for any $\pi \in \B_n$, we have $\des_B(\widetilde{\pi})=n-\des_B(\pi)$.
\end{lem}

\begin{proof}
It is convenient to associate  a permutation in $\B_n$ with  a character string in  $\{+, -\}^n$ by replacing each
 positive (resp., negative) element  with $+$ (resp., $-$).
 For example, the string for permutation $1\,5\,\overline{3}\,4\,6\,\overline{2}$ 
 is  $++-++-$.
Let $\pi \in \B_n$ with $\des_B(\pi)=k$.
We consider the following four cases in terms of the signs of $\pi_1$ and $\pi_n$.
\begin{enumerate}[label=\roman*)]
\item  [\rm (i)] If  $\pi_{1} > 0$ and $\pi_{n} > 0$, then
\begin{equation*}
   |\Pi_1|+|\Pi_2|+|\Pi_3|=k ~\text{and}~|\Pi'_1|+|\Pi'_2|+|\Pi'_3|=n-k-1.
\end{equation*}
In addition, $|\Pi_3|$ (resp., $|\Pi'_3|$) is the number of $+-$ (resp., $-+$) occurring in the character string of  $\pi$.
Obviously, we have $|\Pi_3|=|\Pi'_3|$ since $\pi_{1} > 0$ and $\pi_{n} > 0$.
For the permutation $\widetilde{\pi}=\psi(\pi)$,
it is easy to see that 
$
(n-\Pi'_1)\cup( n-\Pi'_2)\cup(n-\Pi_3)
$
is a subset of descent positions in $\widetilde{\pi}$,
where $n-\Pi$ denotes the set $\{n-i \,|\, i \in \Pi\}$.
Note that $0$ is also a descent position in $\widetilde{\pi}$ since $\widetilde{\pi}_1 =\pi_n-n-1 < 0$, hence
\begin{equation*}
  \des_B(\widetilde{\pi})= 1 + |n-\Pi'_1|+|n-\Pi'_2|+|n-\Pi_3| = 1+ |\Pi'_1|+|\Pi'_2|+|\Pi'_3| =n-k.
\end{equation*}
\item  [\rm (ii)] If $\pi_{1} > 0$ and $\pi_{n} < 0$, then
\begin{equation*}
  |\Pi_1|+|\Pi_2|+|\Pi_3|=k,\, |\Pi'_1|+|\Pi'_2|+|\Pi'_3|=n-k-1~\text{and}~|\Pi_3|=|\Pi'_3|+1.
\end{equation*}
Hence, we have
\begin{equation*}
  \des_B(\widetilde{\pi})= |n-\Pi'_1|+|n-\Pi'_2|+|n-\Pi_3| = |\Pi'_1|+|\Pi'_2|+|\Pi'_3|+1 = n-k.
\end{equation*}
\item [\rm (iii)] If $\pi_{1} < 0$ and $\pi_{n} > 0$, then
\begin{equation*}
  |\Pi_1|+|\Pi_2|+|\Pi_3|=k-1,\,|\Pi'_1|+|\Pi'_2|+|\Pi'_3|=n-k~\text{and}~|\Pi_3|=|\Pi'_3|-1.
\end{equation*}
Note that $0$ is a descent position since $\widetilde{\pi}_1 = \pi_n-n-1 < 0$, hence
\begin{equation*}
  \des_B(\widetilde{\pi})= 1+|n-\Pi'_1|+|n-\Pi'_2|+|n-\Pi_3| = |\Pi'_1|+|\Pi'_2|+|\Pi'_3| = n-k.
\end{equation*}
\item [\rm (iv)] If $\pi_{1} < 0$ and $\pi_{n} < 0$, then
\begin{equation*}
  |\Pi_1|+|\Pi_2|+|\Pi_3|=k-1,\,|\Pi'_1|+|\Pi'_2|+|\Pi'_3|=n-k ~\text{and}~|\Pi_3|=|\Pi'_3|,
\end{equation*}
which implies that
\begin{equation*}
  \des_B(\widetilde{\pi})= |n-\Pi'_1|+|n-\Pi'_2|+|n-\Pi_3| = |\Pi'_1|+|\Pi'_2|+|\Pi'_3| = n-k.
\end{equation*}
\end{enumerate}
Summarising  the above four cases we are done.
\end{proof}

\begin{lem}\label{prop+iden+B+fmaj+PN}
Let $\pi \in \B_n$ and $\Neg(\pi)=m$.
\begin{itemize}
  \item [\rm (a)] If $\pi_{n} < 0$, then $\sum_{i \in \Pi_3}i + m = \sum_{i \in \Pi'_3} i + n$;
  \item [\rm (b)] If $\pi_{n} > 0$, then $\sum_{i \in \Pi_3}i + m = \sum_{i \in \Pi'_3} i$.
\end{itemize}
\end{lem}

\begin{proof}
Let $\pi \in \B_n$ with $\Pi_3=\{i_1, i_2, \dots, i_{\ell}\}$ and $\Pi'_3=\{j_1, j_2, \dots, j_{r}\}$ for some integers $\ell,r \geq 1$.
As the proof of (b) is similar,  we only prove (a) by considering two cases.
\begin{enumerate}[label=\roman*)]
\item [\rm (i)]  $\pi_{1} > 0$ and $\pi_{n} < 0$, we have $\ell=r+1$.
       It is easy to see that $i_{k+1}-j_{k}$ is the number of positive elements between the $k$th ascent position and the $(k+1)$th descent position from left to right.
       Note that $|\Pi_3|=|\Pi'_3|+1$ in this case.
       Therefore, we have $i_1 + \sum_{k=1}^{r}(i_{k+1}-j_{k}) = n-m$.
       
\item [\rm (ii)] $\pi_{1} < 0$ and $\pi_{n} < 0$, we have $\ell=r$.
       Similarly, $i_{k}-j_{k}$ is the number of positive elements between the $k$th ascent position and the $k$th descent position from left to right.
       Then we have $\sum_{k=1}^{r}(i_{k}-j_{k}) = n-m$.
\end{enumerate}
Combining the above two cases completes the proof of (a).
\end{proof}

The following \emph{$q$-symmetry} of $B_{n,k}(q)$ is crucial for our 
combinatorial proof of identity~\eqref{eq+eulerian+stir+B+q+trans}.
\begin{prop}\label{prop+iden+B+fmaj}
For each fixed nonnegative integer $n$ and the polynomial $B_{n,k}(q)$ defined in~\eqref{def+q+eulerian+poly+B}, we have
\begin{equation}\label{q-symmetry}
  B_{n,k}(q) = q^{2nk-n^2}B_{n,n-k}(q)
\end{equation}
for $0 \leq k \leq n$.
\end{prop}
\begin{proof}%[Proof of Proposition~\ref{prop+iden+B+fmaj}]
For any $\pi \in \B_n$ with $k$ descents and $m$ negative elements,
then $\des_B(\psi(\pi))=n-k$ by Lemma~\ref{prop+bijection+B}.
Hence, it suffices to show that
\begin{equation*}
  \fmaj(\pi) = 2nk-n^2+\fmaj(\psi(\pi)).
\end{equation*}
Let $\widetilde{\pi}=\psi(\pi)$, we consider the proof in terms of the signs of $\pi_1$ and $\pi_n$.
We only give the proof for this case $\pi_{1} > 0$ and $\pi_{n} > 0$ and omit similar discussions for
other three cases for the brevity.

If $\pi_{1} > 0$ and $\pi_{n} > 0$, by the definition of the mapping $\psi$,
then the set of descents in $\widetilde{\pi}$ is the disjoint union
$$
\{0\}\cup (n-\Pi'_1)\cup (n-\Pi'_2)\cup (n-\Pi_3)
$$ 
and $\widetilde{\pi}$ has $n-m$ negative elements.
Then,
\begin{equation}\label{iden+fmaj+map+pi}
    \fmaj(\widetilde{\pi})
  = 2\left(\sum_{i \in n-\Pi'_1} i + \sum_{i \in n-\Pi'_2} i+\sum_{i \in n-\Pi_3} i\right) + n-m.
\end{equation}
By Case (i) in the proof of Proposition~\ref{prop+bijection+B}, we have $|\Pi'_1|+|\Pi'_2|+|\Pi_3|=n-k-1$.
Hence identity~\eqref{iden+fmaj+map+pi} is equivalent to
\begin{align*}
      \fmaj(\widetilde{\pi})
  &= 2n(n-k-1)- 2\left(\sum_{i \in \Pi'_1} i + \sum_{i \in \Pi'_2} i+\sum_{i \in \Pi_3} i\right) + n-m              \\
  &= 2n(n-k-1)- 2\left(\binom{n}{2}-\sum_{i \in \Pi_1} i - \sum_{i \in \Pi_2} i-\sum_{i \in \Pi'_3} i\right) + n-m \\
  &= n^2-2nk+2\sum_{i \in \Pi_1} i + 2\sum_{i \in \Pi_2} i+2\sum_{i \in \Pi'_3} i - m,
\end{align*}
where the second equality uses the fact that the sum of all descent and ascent indexes is $\binom{n}{2}$.
By statement (b) of Proposition~\ref{prop+iden+B+fmaj+PN}, the above identity equals
\begin{align*}
      \fmaj(\widetilde{\pi})
  &= n^2-2nk+2\sum_{i \in \Pi_1} i +  2\sum_{i \in \Pi_2} i+2\sum_{i \in \Pi_3} i + m \\
  &= n^2-2nk+\fmaj(\pi).
\end{align*}
This is the desired result.
\end{proof}

\subsection{Ordered set partitions of type B}\label{subsection+B+order+partition}
Recall that $\langle n\rangle=\{\overline{n},\ldots,\overline{1},0,1,\ldots,n\}$.
There are at least two  equivalent definitions of type B set partition.  We say that a  set partition of  $\langle n\rangle$
is a \emph{type B partition} if it satisfies the following properties
\begin{itemize}
\item [\rm (1)] there exactly is one zero block $T$ such that $0\in T$ and $-T=T$;
\item [\rm (2)] if $T$ appears as a block then $-T$ is also a block.
\end{itemize}

It is known \cite{BBG19, SS22} that  $S_B(n,k)$ is the number of type B partitions of $\langle n\rangle$ with $2k+1$ blocks. An \emph{ordered signed partition} of $\langle n\rangle$
is a sequence $(T_0,T_1,T_2,\ldots,T_{2k})$ of disjoint subsets (blocks) $T_i$ of $\langle n\rangle$
satisfying
\begin{itemize}
  \item [\rm (1)] $0 \in T_0$ and $T_0=\overline{T}_0$, and
  \item [\rm (2)]  $T_{2i}=\overline{T}_{2i-1}$ for $i\in [k]$,
\end{itemize}
where $\overline{T}=\{\overline{t}: t \in T\}$.
The blocks $T_{2i}$ and $T_{2i-1}$ are called \emph{paired}.
Clearly the number of all ordered signed partitions of $\langle n\rangle$ with $2k+1$ blocks is $2^kk!S_B(n,k)$.

For our purpose, it is convenient to use the following equivalent definition of ordered signed partition.
An \emph{ordered set partition with sign} of $S=\{0,1,\ldots,n\}$ is a sequence $(S_0, S_1, \ldots, S_k)$ such that
\begin{enumerate}
\item  [\rm (1)] $S_0=\{t\in T_0: t\leq 0\}$, and
\item  [\rm (2)] $S_i=T_{2i-1}$ for $i\in [k]$.
\end{enumerate}
For example, the sequence$(\{0, \overline{3}, \overline{1}, \overline{4}\}, \{\overline{2}, 7\}, \{\overline{6}\},\{8, \overline{5}\})$
is an ordered set partition with sign of $\{0,1,\ldots,8\}$.

On the other hand, as in \cite{RW15},
we can consider an ordered set partition with sign as a descent-starred signed permutation,
i.e, for any $\pi \in \B_n$, the space following element $\pi_i$,
satisfying $\pi_i > \pi_{i+1}$ for some $0 \leq i \leq n-1$, is starred or unstarred.
That is to say, instead of using brackets to signify separations between blocks,
the spaces between elements sharing a block can be marked with stars and all blocks are written in decreasing order.
%And then all positive elements $i$ in the block containing $0$ will be replaced by $\overline{i}$.
Note that we require that the block including element $0$ always stands first on the list.
%and the element $0$ is in the first place in the block containing it.

For example, the ordered set partition with sign $(\{0,\overline{3},\overline{1},\overline{4}\},\{\overline{2},7\},\{\overline{6}\},\{8,\overline{5}\})$
can be written as $0_{*}\overline{1}_{*}\overline{3}_{*}\overline{4}~7_{*}\overline{2}~\overline{6}~8_{*}\overline{5}$.
The above discussion shows that there is a bijection between all ordered set partitions with sign of the set $\{0,1,\ldots,n\}$
and all descent-starred signed permutations in $\B_n$.
For $0\leq k\leq n$ define the set
\begin{equation}
  \B^{>}_{n,k}:=\{(\pi,S): \pi \in \B_n, S \subseteq \Des_B(\pi), |S|=k \},
\end{equation}
where $S$ is the set of the starred descent positions.

For $(\pi,S) \in \B^{>}_{n,k}$ define the statistic
\begin{equation*}
  \fmaj((\pi, S)):=\fmaj(\pi)-\sum_{j \in S}(2|\Des_B(\pi)\cap\{j,\ldots,n-1\}|-1)
\end{equation*}
and the polynomial
\begin{equation}\label{def+fmaj+B}
  B^{\fmaj}_{n,k}(q):=\sum_{(\pi,S) \in \B^{>}_{n,k}}q^{\fmaj((\pi,S))}.
\end{equation}
By the definition of the statistic $\fmaj((\pi,S))$,
we attach the $i$th descent position of $\pi$ (from right to left) with the weight $1$
if this descent position is unstarred and the weight $z/q^{2i-1}$ if this descent position is starred.
Therefore, the following identity holds
\begin{equation}\label{ide+fmaj+star}
     \sum_{k=0}^{n}B^{\fmaj}_{n,k}(q)z^k
   = \sum_{\pi \in {\B_n}}q^{\fmaj(\pi)}\prod\limits_{i=1}^{\des_B(\pi)}\left(1+\frac{z}{q^{2i-1}} \right).
\end{equation}

For convenience, we  recall two known $q$-identities (see \cite[Theorem 3.3]{An76})
\begin{align}
    \prod_{i=1}^N(1-zq^{i-1})
 &= \sum_{j=0}^N {N\brack j}_q(-1)^j z^j q^{j(j-1)/2}; \label{q1}\\
    \frac{1}{\prod_{i=1}^N(1-zq^{i-1})}
 &= \sum_{j=0}^{\infty} {N+j-1\brack j}_q z^j.         \label{q2}
\end{align}

%Next, using the following fact
%\begin{equation*}
%  2k(k-\ell)-2\binom{k}{2}+2n\ell-2\binom{n+1}{2}=2\binom{n-k}{2}-2(n-k)(n-\ell),
%\end{equation*}
We first establish  the following result for polynomials  $B_{n,n-\ell}(q)$ and $B^{\fmaj}_{n,n-k}(q)$
 defined by~\eqref{def+q+eulerian+poly+B} and~\eqref{def+fmaj+B}.
\begin{prop}\label{prop+fmaj+Euler+B}
For $0 \leq k \leq n$ we have
\begin{equation*}
    B^{\fmaj}_{n,n-k}(q)
  = \sum_{\ell=0}^{k}q^{(n-k)(2\ell-n-k)}B_{n,n-\ell}(q)
    {n-\ell \brack k-\ell}_{q^2}.
\end{equation*}
\end{prop}

\begin{proof}
Let $(\pi, S) \in \B^{>}_{n,n-k}$, then there are $n-k$ starred descents in $(\pi, S)$,
this  means that the number of ascents is in $\{0\} \cup [k]$.
Suppose that the signed permutation $\pi$ has $\ell$ ascents, where $\ell \in \{0\} \cup [k]$,
then the signed permutation $\pi$ can be any permutation in $\B_n$ with $n-\ell$ descents.
Therefore, the sum of $q$-counting
about the flag-major statistic for all possible signed permutations with $n-\ell$ descents is the polynomial $B_{n,n-\ell}(q)$.

In addition, for a signed permutation $\pi$ with $n-\ell$ descents,
we can choose $n-k$ descents from $n-\ell$ descents in $\pi$ and mark them with stars.
By the definition of the statistic $\fmaj((\pi,S))$ and identities~\eqref{ide+fmaj+star} and \eqref{q1},
we have
\begin{align*}
    [z^{n-k}]\prod
    \limits_{i=1}^{n-\ell}
    \left(1+\frac{zq}{q^{2i}}\right)
  = q^{(n-k)(2\ell-n-k)}{n-\ell \brack n-k}_{q^2},
\end{align*}
where $[z^{k}]f(z)$ denotes the coefficient of $z^{k}$ in the polynomial $f(z)$.
Using the symmetry of $q$-binomial coefficients
\begin{equation*}
  {n-\ell \brack n-k}_{q^2}={n-\ell \brack k-\ell}_{q^2},
\end{equation*}
we complete the proof.
\end{proof}

To derive  a  recurrence relation for  the polynomials $B^{\fmaj}_{n,n-k}(q)$,
we introduce some notations.
For other unstarred positions, we label the rightmost position in our descent-starred signed permutation with $0$,
and then label its unlabelled descent positions from right to left with 1,2,\ldots.
Next, all other unlabelled positions from left to right are labelled with increasing labels starting from the next number.
We call the above labelling as \emph{fmaj-labelling}.
For example, if $(\pi, S) = 4_{*}3_{*}\overline{1}~7_{*}\overline{2}~\overline{6}~8_{*}\overline{5}$,
then the fmaj-labelling for $(\pi, S)$~is
\begin{equation*}
	_{2} 4_{*}3_{*}\overline{1}_{3} 7_{*}\overline{2}_{1} \overline{6}_{4} 8_{*}\overline{5}_{0}.
\end{equation*}

For $\alpha=n$ or $\overline{n}$, we define the mapping
\begin{equation}\label{def:phi1}
  \phi^{|}_{\alpha,k}:\{0,1,\ldots,n-k-1\}\times \mathcal{B}^{>}_{n-1,k} \rightarrow \mathcal{B}^{>}_{n,k}
\end{equation}
by sending $(i,(\pi,S))$ to the descent-starred signed permutation obtained from $(\pi,S)$ by
\begin{itemize}
  \item [\rm (1)] inserting $\alpha$ at the fmaj-labelling $i$, and then
  \item [\rm (2)] moving each star on the right of $\alpha$ one descent to its left.
\end{itemize}
Clearly, the rightmost descent will be unstarred when the letter $n$ is not inserted after $\pi_{n-1}$.
Thus, we have the following relation between these labels and insertion mappings.

\begin{lem}\label{lem+fmaj+bar}
For $0 \leq k \leq n-1$ we have
\begin{itemize}
\item [\rm (a)] if $(\pi, S) \in \mathcal{B}^{>}_{n-1, k}$, then $\fmaj(\phi^{|}_{n,k}(i,(\pi,S)))=\fmaj((\pi,S))+2i$
                for $i \in \{0\} \cup [n-k-1]$;
\item [\rm (b)] if $(\pi, S) \in \mathcal{B}^{>}_{n-1, k}$, then $\fmaj(\phi^{|}_{\overline{n},k}(i,(\pi,S)))=\fmaj((\pi,S))+2i-1$
                for $i \in [n-k-1]$;
\item [\rm (c)] if $(\pi, S) \in \mathcal{B}^{>}_{n-1, k}$, then $\fmaj(\phi^{|}_{\overline{n},k}(0,(\pi,S)))=\fmaj((\pi,S))+2n-2k-1$.
\end{itemize}
\end{lem}

\begin{proof}
We will discuss the change of the statistic $\fmaj((\pi, S))$ in terms of the insertion position of $n$ or $\overline{n}$.
Suppose that the space labelled $i$ under the fmaj-labelling of $(\pi, S)$ is the space immediately following $\pi_p$.
Moreover, we suppose that there are a starred descents and b unstarred descents to the left of $\pi_p$
and $c$ unstarred descents and d starred descents to the right of $\pi_{p+1}$ in $(\pi, S)$.

For (a), inserting $n$ into the space labelled $i$. Let
$(\tau, T)=\phi^{|}_{n,k}(i,(\pi,S))$.
If $i=0$, that is to say we insert $n$ at the end, then the insertion of
$n$ does not affect $\fmaj((\pi, S))$,
thus $\fmaj((\tau, T))=\fmaj((\pi, S))$.
For $i \neq 0$, there will exist two cases in terms of the values of $\pi_p$ and $\pi_{p+1}$.

Case (i):
If $\pi_p > \pi_{p+1}$, then $i=c+1$. By inserting $n$ after $\pi_p$,
which preserves each descent position before $\pi_p$ and increases each descent position after $\pi_{p}$ by one.
Thus, the statistic $\fmaj(\tau)=\fmaj(\pi)+2c+2d+2$.
In addition, the insertion of $n$ does not affect the starred descents before $\pi_p$ to the corresponding sum
$\sum_{j \in S}(2|\Des_B(\pi)\cap\{j,\ldots,n-2\}|-1)$.
Moving each star after $\pi_{p+1}$ one descent to its left
that increases the sum $\sum_{j \in S}(2|\Des_B(\pi)\cap\{j,\ldots,n-2\}|-1)$ by two.
Therefore, we have
\begin{equation*}
  \sum_{j \in T}(2|\Des_B(\tau)\cap\{j,\ldots,n-1\}|-1)=\sum_{j \in S}(2|\Des_B(\pi)\cap\{j,\ldots,n-2\}|-1)+2d
\end{equation*}
since there are d stars after $\pi_{p+1}$. Hence,
\begin{align*}
     \fmaj((\tau, T))
  &= \fmaj(\tau)-\sum_{j \in T}(2|\Des_B(\tau)\cap\{j,\ldots,n-1\}|-1)                          \\
  &= \fmaj(\pi)+2c+2d+2-\sum_{j \in S}(2|\Des_B(\pi)\cap\{j,\ldots,n-2\}|-1)-2d    \\
  &= \fmaj((\pi, S))+2c+2                                                                     \\
  &= \fmaj((\pi, S))+2i
\end{align*}
for $i \in [n-k-1]$.

Case (ii):
If $\pi_p < \pi_{p+1}$, then $i=p+1-a+c$. By inserting $n$ after $\pi_p$,
which preserves each descent position before $\pi_p$ and increases each descent position after $\pi_{p}$ by one.
Besides, note that there is a new descent, $p+1 \in \Des_B(\tau)$ while inserting $n$ after $\pi_p$.
Thus, the statistic $\fmaj(\tau)=\fmaj(\pi)+2p+2+2c+2d$.
In addition, the insertion of $n$ increases each starred descent before $\pi_p$ to the corresponding sum
$\sum_{j \in S}(2|\Des_B(\pi)\cap\{j,\ldots,n-2\}|-1)$ by two.
Moving each star after $\pi_{p+1}$ one descent to its left
that increases the sum $\sum_{j \in S}(2|\Des_B(\pi)\cap\{j,\ldots,n-2\}|-1)$ by two.
Therefore,
\begin{equation*}
  \sum_{j \in T}(2|\Des_B(\tau)\cap\{j,\ldots,n-1\}|-1)=\sum_{j \in S}(2|\Des_B(\pi)\cap\{j,\ldots,n-2\}|-1)+2a+2d
\end{equation*}
since there are a stars before $\pi_p$ and d stars after $\pi_{p+1}$. Hence,

\begin{align*}
     \fmaj((\tau, T))
  &= \fmaj(\tau)-\sum_{j \in T}(2|\Des_B(\tau)\cap\{j,\ldots,n-1\}|-1)          \\
  &= \fmaj(\pi)+2p+2+2c+2d                                                      \\
  &\quad-\sum_{j \in S}(2|\Des_B(\pi)\cap\{j,\ldots,n-2\}|-1)-2a-2d             \\
  &= \fmaj((\pi, S))+2p+2-2a+2c                                                 \\
  &= \fmaj((\pi, S))+2i
\end{align*}
for $i \in [n-k-1]$.

For (b), inserting $\overline{n}$ into the space labelled $i$.
Let $(\mu, R)=\phi^{|}_{\overline{n},k}(i,(\pi,S))$.
For $i \neq 0$, all changes for $\fmaj((\mu, R))$ are the same to (a) except that
for the statistic $\fmaj(\pi)$ when the new descent position generated by $\overline{n}$.
In this case, there always exists one descent between $\pi_p$ and $\overline{n}$.
The descent generated by $\overline{n}$ increases the statistic
$\fmaj(\pi)$ by $1$ when $\pi_p > \pi_{p+1}$ and $2p+1$ when $\pi_p < \pi_{p+1}$, respectively.
For the insertion of $n$ at same position, the changes separately are $2$ and $2p+2$ for those two cases.
Following the discussion of (a), it is easy to know that
\begin{equation*}
    \fmaj((\mu, R))=\fmaj((\pi, S))+2i-1
\end{equation*}
for $i \in [n-k-1]$.

For (c), if $i=0$, inserting $\overline{n}$ after $\pi_{n-1}$,
then the only change is the new descent $\pi_{n-1} > \overline{n}$.
That is to say, the insertion of $\overline{n}$ increases $\fmaj(\pi)$
and $\sum_{j \in S}(2|\Des_B(\pi)\cap\{j,\ldots,n-2\}|-1)$ by $2n-1$ and $2k$, respectively.
Thus,
\begin{equation*}
    \fmaj((\mu, R))=\fmaj((\pi, S))+2n-2k-1.
\end{equation*}
Summarising the above cases we have completed the proof.
\end{proof}

As mentioned before, the mapping $\phi^{|}_{\alpha,k}$ preserves the number of stars in the mapping process.
Similarly, we need to define some mappings that increase the number of stars by one as follows:
\begin{equation}\label{def:phi*}
 \phi^{*}_{n,k}:\{1,2,\ldots,n-k\}\times \B^{>}_{n-1,k-1} \rightarrow \B^{>}_{n,k}
 \end{equation}
and
\begin{equation}
 \phi^{*}_{\overline{n},k}:\{0,1,\ldots,n-k\}\times \B^{>}_{n-1,k-1} \rightarrow \B^{>}_{n,k},
\end{equation}
which send $(i,(\pi,S))$ to the descent-starred signed permutation obtained from $(\pi,S)$ by
\begin{itemize}
  \item [\rm (1)] inserting $n$ (resp., $\overline{n}$) at the fmaj-labelling $i$, then
  \item [\rm (2)] moving each star on the right of $n$ (resp., $\overline{n}$) one descent to its left, and then
  \item [\rm (3)] placing a star at the rightmost descent of the resulting descent-starred signed permutation.
\end{itemize}

In analogy with the discussion in the proof of Lemma \ref{lem+fmaj+bar},
let $\alpha=n$ or $\overline{n}$ and $(\tau, T)=\phi^{*}_{\alpha,k}(i,(\pi,S))$.
The first step and second one from the mapping $\phi^{*}_{\alpha,k}$
have same effect with $\phi^{|}_{\alpha,k}$ to the statistics $\fmaj(\pi)$
and $\sum_{j \in S}(2|\Des_B(\pi)\cap\{j,\ldots,n-2\}|-1)$.
The last step from the mapping $\phi^{*}_{\alpha,k}$,
placing a star at the rightmost of resulting descent-starred signed permutation,
which increases the sum $\sum_{j \in S}(2|\Des_B(\pi)\cap\{j,\ldots,n-2\}|-1)$ by one.
Therefore, we have the following results, of which the proof is omitted for the brevity.

\begin{lem}\label{lem+fmaj+star}
For $1 \leq k \leq n$ we have
\begin{itemize}
\item [\rm (a)] if $(\pi, S) \in \B^{>}_{n-1, k-1}$, then $\fmaj(\phi^{*}_{n,k}(i,(\pi,S)))=\fmaj((\pi,S))+2i-1$
                for $i \in [n-k]$;
\item [\rm (b)] if $(\pi, S) \in \B^{>}_{n-1, k-1}$, then $\fmaj(\phi^{*}_{\overline{n},k}(i,(\pi,S)))=\fmaj((\pi,S))+2i-2$
                for $i \in [n-k]$;
\item [\rm (c)] if $(\pi, S) \in \B^{>}_{n-1, k-1}$, then $\fmaj(\phi^{*}_{\overline{n},k}(0,(\pi,S)))=\fmaj((\pi,S))+2n-2k$.
\end{itemize}
\end{lem}

By definitions \eqref{def:phi1} and \eqref{def:phi*},
the mappings $\phi^{|}_{\alpha,k}$ and $\phi^{*}_{\alpha,k}$ ($\alpha=n$ or $\overline{n}$) have
their images $\I_0\cup \I_1$ and $\I_2$, respectively, where
\begin{align*}
\I_0&=\{(\pi, S) \in \B^{>}_{n,k}: \pi_n=n\};\\
\I_1&=\{(\pi, S) \in \B^{>}_{n,k}: \text{rightmost descent is unstarred in}~(\pi, S)~\text{and}~ \pi_n \neq n\};\\
\I_2&=\{(\pi, S) \in \B^{>}_{n,k}: \text{rightmost descent is starred in}~ (\pi, S)~\text{and}~ \pi_n \neq n\}.
\end{align*}
Obviously, the disjoint union of those three sets is $\B^{>}_{n,k}$.
Now, we are ready to prove the following recurrence relation for the polynomial $B^{\fmaj}_{n,k}(q)$ defined in~\eqref{def+fmaj+B}.
\begin{prop}\label{prop+rec+maj+B}
For $n\geq 1$ we have the recurrence relation
\begin{equation*}
  B^{\fmaj}_{n,k}(q)=[2n-2k]_q\,B^{\fmaj}_{n-1,k}(q)+[2n-2k+1]_q\,B^{\fmaj}_{n-1,k-1}(q),
\end{equation*}
where $B^{\fmaj}_{n,k}(q)$ is $1$ when $k=n$ and is $0$ when $k<0$ or $k>n$.
\end{prop}

\begin{proof}
Since $\B^{>}_{n,k}$ is the disjoint union of the images of mappings $\phi^{|}_{\alpha,k}$ and $\phi^{*}_{\alpha,k}$, we have
\begin{align}\label{eq+fmaj+decom+map}
     B^{\fmaj}_{n,k}(q)
  &= \sum_{(\pi, S)\in \I_0}q^{\fmaj((\pi, S))}
    +\sum_{(\pi, S)\in \I_1}q^{\fmaj((\pi, S))}
    +\sum_{(\pi, S)\in \I_2}q^{\fmaj((\pi, S))}.
\end{align}
By the definition of mapping $\phi^{|}_{\alpha,k}$ and Lemma \ref{lem+fmaj+bar},
the first two summations of identity~\eqref{eq+fmaj+decom+map} is
\begin{align}\label{eq+fmaj+decom+bar}
 &\quad  \sum_{(\pi, S)\in \I_0}q^{\fmaj((\pi, S))}
     +\sum_{(\pi, S)\in \I_1}q^{\fmaj((\pi, S))} \nonumber\\
 &=   \sum_{i=0}^{n-k-1}\sum_{(\pi,S)\in \B^{>}_{n-1,k}}q^{\fmaj(\phi^{|}_{n,k}(i,(\pi,S)))}
     +\sum_{i=0}^{n-k-1}\sum_{(\pi,S)\in \B^{>}_{n-1,k}}q^{\fmaj(\phi^{|}_{\overline{n},k}(i,(\pi,S)))}   \nonumber\\
 &=   \sum_{(\pi,S)\in \B^{>}_{n-1,k}}q^{\fmaj((\pi,S))}
      \left(\sum_{i=0}^{n-k-1}q^{2i}+\sum_{i=1}^{n-k-1}q^{2i-1}+q^{2n-2k-1}\right)                        \nonumber\\
 &= [2n-2k]_q\,B^{\fmaj}_{n-1,k}(q).
\end{align}

Similarly, by the definition of mapping $\phi^{*}_{\alpha,k}$ and Lemma \ref{lem+fmaj+star},
the last summation of identity~\eqref{eq+fmaj+decom+map} is
\begin{align}\label{eq+fmaj+decom+star}
  \sum_{(\pi, S)\in \I_2}q^{\fmaj((\pi, S))}
  & = \sum_{i=1}^{n-k}\sum_{(\pi,S)\in \B^{>}_{n-1,k-1}}q^{\fmaj(\phi^{*}_{n,k}(i,(\pi,S)))}
     +\sum_{i=0}^{n-k}\sum_{(\pi,S)\in \B^{>}_{n-1,k-1}}q^{\fmaj(\phi^{*}_{\overline{n},k}(i,(\pi,S)))}      \nonumber\\
  & = \sum_{(\pi,S)\in \B^{>}_{n-1,k-1}}q^{\fmaj((\pi,S))}
      \left(\sum_{i=1}^{n-k}q^{2i-1}+\sum_{i=1}^{n-k}q^{2i-2}+q^{2n-2k} \right)                             \nonumber\\
  &=  [2n-2k+1]_q\,B^{\fmaj}_{n-1,k-1}(q).
\end{align}\label{eq+fmaj+decom+star+fina}
Combining~\eqref{eq+fmaj+decom+map}-\eqref{eq+fmaj+decom+star} completes the  proof.
\end{proof}

%Now, we are ready to prove  Theorem~\ref{thm+stirling+frobenius+B}.

\begin{proof}[\bf Proof of Theorem~\ref{thm+stirling+frobenius+B}]
By Proposition~\ref{prop+iden+B+fmaj} we can rewrite
identity~\eqref{eq+eulerian+stir+B+q+trans} as
\begin{equation}\label{eq+stirling+B+Euler}
    [2]_q^k[k]_{q^2}!S_B[n,k]
  = \sum_{\ell=0}^{k}q^{(n-k)(2\ell-n-k)}B_{n,n-\ell}(q)
    {n-\ell \brack k-\ell}_{q^2}.
\end{equation}
Let $S^{o}_B[n,k]$ be the left-hand side of \eqref{eq+stirling+B+Euler}.
It follows from  Eq.~\eqref{def+stirling+B+q} that the sequence $(S^{o}_B[n,k])_{0\leq k\leq n}$
is determined by the recurrence relation
\begin{equation}\label{def+stirling+B+q+order}
     S^{o}_B[n,k]
  := [2k]_q\,S^{o}_B[n-1,k-1] + [2k+1]_q\,S^{o}_B[n-1,k]
\end{equation}
with  $S^{{o}}_B[0,k]=\delta_{0k}$.
Invoking Proposition~\ref{prop+rec+maj+B} we see that 
the polynomials $B^{\fmaj}_{n,n-k}(q)$ satisfy recurrence relation~\eqref{def+stirling+B+q+order}, namely
\begin{equation*}
  B^{\fmaj}_{n,n-k}(q)=S^{o}_B[n,k].
\end{equation*}
Combining with Proposition~\ref{prop+fmaj+Euler+B},
we have  a combinatorial proof of~\eqref{eq+stirling+B+Euler}.
\end{proof}

\section{$q$-Stirling numbers of the second kind in type D}\label{section+generalization}

Recently, Bagno et al. \cite{BBG19} studied some identities about the type D Stirling numbers of the second kind $S_D(n,k)$. 
As far as we know, there is no $q$-Stirling numbers of the second kind in type D in the literature.
In this section, we first define a $q$-Stirling numbers of the second kind in type D
and prove $q$-analogues of two known results about the Stirling numbers of the second kind in types A, B and D,
see Proposition \ref{prop+rel+stirling+q+A+B+D}.
Then, we establish a $q$-identity connecting the $q$-falling factorials of type D and the $q$-Stirling numbers of the second kind in type D, see Proposition \ref{prop+stir+D+q}.
%These identities are $q$-analogues of known formulae in \cite{Za81,Sut00,BBG19}.

\subsection{Two $q$-identities about the $q$-Stirling numbers of the second kind}

Using the definitions and notations of ordered signed partition in Subsection \ref{subsection+B+order+partition},
we say that the set $\{T_0,T_1,T_2,\ldots,T_{2k}\}$ is a \emph{signed partition} of $\langle n\rangle$
if $(T_0,T_1,T_2,\ldots,T_{2k})$ is an ordered signed partition.
A signed partition $\pi=\{T_0,T_1,T_2,\ldots,T_{2k}\}$ of $\langle n\rangle$ is
called \emph{type D} if $\#T_0\neq 3$, where $\#T$ denotes the cardinality of a finite set $T$,
in other words, the block $T_0$ contains at least two positive elements or only contains 0.
Let $S_D(n,k)$ be the number of all type D signed partitions of $\langle n\rangle$ with $2k+1$ blocks,
see an equivalent definition of $S_D(n,k)$ in \cite{BBG19}.
The numbers $S_D(n,k)$ are called the \emph{Stirling numbers of the second kind in type D}.

For $0 \leq k \leq n$, the following two identities about
the Stirling numbers of the second kind in types A, B and D were implicitly given
in~\cite[Corollary 12]{Za81}, \cite[Eq. (19)]{CG07} and~\cite[Proposition $3$]{Sut00}:
\begin{align}
     S_B(n,k) &=\sum_{j=k}^{n}2^{j-k}\binom{n}{j}S(j,k); \label{rel+stirling+AB}  \\
     S_B(n,k) &= S_D(n,k)+n\cdot2^{n-k-1}S(n-1,k).       \label{rel+stirling+ABD}
\end{align}
%Next, we will present a $q$-analogue of identity \eqref{rel+stirling+ABD}.
In this subsection, we define a kind of type D $q$-Stirling numbers of the second kind $S_D[n,k]$,
and give $q$-analogues of identities~\eqref{rel+stirling+AB} and \eqref{rel+stirling+ABD}.

\begin{defn}
For any $S \subset \Z\backslash\{0\}$ let $\overline{S}=\{\overline{i}: i\in S\}$.
A \emph{standard signed partition} ({\sc SSP} for short) of $S$ is
a sequence $\pi=(S_1,S_2,\ldots,S_k)$ of disjoint nonempty subsets of $S\cup\overline{S}$ such that
\begin{enumerate}
\item [\rm (1)] $\{S_1,\ldots,S_k,\overline{S}_1,\ldots,\overline{S}_k\}$ is a partition
of $S\cup\overline{S}$;
\item [\rm (2)]  $\min|S_1| \leq \min|S_2| \leq \cdots \leq \min|S_k|$, where  $|S_i|=\{|j|:j\in S_i\}$ for $i\in [k]$.
\end{enumerate}
The sets $S_1,S_2,\ldots,S_k$ are the \emph{blocks} of $\pi$ (so $\pi$ has $k$ blocks).
A \emph{partial standard signed partition} ({\sc PSSP} for short) of $S$ is a standard signed partition of a subset of $S$.
\end{defn}

Let $B(S,k)$ (resp., $B_{\subseteq}(S,k)$) be the set of all {\sc SSP} (resp., {\sc PSSP}) of $S$ with $k$ blocks.
Let $D_{\subseteq}([n],k)$ denote the set of all {\sc PSSP} of $[n]$ that
excludes all {\sc SSP} of $[n]\backslash\{i\}$ with $k$ blocks for $i \in [n]$, namely,
\begin{equation*}
    D_{\subseteq}([n],k)
  = B_{\subseteq}([n],k) \backslash \bigcup_{i=1}^{n}B([n]\backslash \{i\},k).
\end{equation*}

\begin{lem}\label{lem+type+D}
For $0 \leq k \leq n$ we have
\begin{equation*}
2^kS_D(n,k)=\#D_{\subseteq}([n],k).
\end{equation*}
\end{lem}

\begin{proof}
%As mentioned at the beginning of Subsection~\ref{subsection+B+order+partition},
For any {\sc PSSP} $\pi=(T_1, T_2, \ldots, T_k) \in D_{\subseteq}([n],k)$,
it is clear that the sequence $(\langle n \rangle \backslash \{T\cup\overline{T}\}, T_1, \overline{T}_1, \ldots, T_k,\overline{T}_k)$
is an ordered signed partition of the set $\langle n\rangle$, where $T=\cup_{i=1}^kT_i$.
Thus, the set
\begin{equation*}
\Pi=  \left\{\langle n \rangle \backslash \{T\cup\overline{T}\},
      T_1, \overline{T}_1, \ldots, T_k,\overline{T}_k\right\}
\end{equation*}
is a type D signed partition of $\langle n \rangle$.
Due to the choice of $T_i$ and $\overline{T_i}$,
both {\sc PSSP} $\pi=(T_1, \ldots, T_i, \ldots, T_k)$
and $\pi'=(T_1, \ldots, \overline{T_i}, \ldots, T_k)$ correspond to the type D signed partition $\Pi$,
which implies the desired result.
\end{proof}

\begin{defn}
For $\pi=(S_1, S_2, \ldots, S_k)\in B_{\subseteq}(S,k)$, define the statistics
\begin{equation*}
  \pos(\pi):=\#\left\{ x \in \bigcup_{i=1}^k  S_i: x > 0 \right\}
\end{equation*}
and
\begin{equation}\label{def+m}
  m(\pi) := 2\sum_{i=1}^ki\cdot\#S_i-\pos(\pi).
\end{equation}
\end{defn}
The following result was incorrectly stated in
\cite[Proposition 4.2]{CG07} 
with 
$
m(\pi) = 2\sum_{i=1}^k(i-1)\#S_i+n+1-\pos(\pi).
$
For completeness, we reproduce  their proof with correction.
\begin{prop}\label{prop+stir+B+q+iden}
Let $m(\pi)$ be defined by \eqref{def+m}. Then we have
\begin{equation}\label{eq=CG}
  q^{k^2}[2]_q^{k}S_B[n,k]=\sum_{\pi \in B_{\subseteq}([n],k)}q^{m(\pi)}
\end{equation}
for $0 \leq k \leq n$.
\end{prop}
\begin{proof}
Let
\begin{equation*}
S_B(n,k,q)=\sum_{\pi \in B_{\subseteq}([n],k)}q^{m(\pi)}.
\end{equation*}
By recurrence \eqref{def+stirling+B+q} of $S_B[n,k]$,
it suffices to show that $S_B(n,k,q)$ satisfies
\begin{equation*}
S_B(n,k,q)=q^{2k-1}(1+q)S_B(n-1,k-1,q)+[2k+1]_qS_B(n-1,k,q)
\end{equation*}
with the initial conditions $S_B(0,k,q)=\delta_{0k}$ for $q\neq 0$.
The case $n=0$ is trivial.
Suppose that $n > 0$ and $\pi=(T_1,\ldots,T_k) \in B_{\subseteq}([n],k)$.
If $\{n\}$ (resp., $\{-n\}$) is a block of $\pi$,
then $\{n\}=T_k$ (resp., $\{-n\}=T_k$) and removing it from $\pi$ yields
a {\sc PSSP} $\tau$ of $[n-1]$ into $k-1$ blocks,
such that  $\pos(\tau)=\pos(\pi)-1$ (resp., $\pos(\tau)=\pos(\pi)$)
and  $m(\pi)=m(\tau)+2k-1$ (resp., $m(\pi)=m(\tau)+2k$)~\footnote{In the proof of 
\cite[Proposition 4.2]{CG07} 
with $m(\pi) = 2\sum_{i=1}^k(i-1)\#S_i+n+1-\pos(\pi)$ the equation  $m(\pi)=m(\tau)+2k-1$ (resp., $m(\pi)=m(\tau)+2k$) does not hold.}.

If $n$ is an element of $T_i$ for some $i \in [k]$, then removing it from $T_i$ yields
a {\sc PSSP} $\tau'$ of $[n-1]$ into $k$ blocks such that  $\pos(\tau')=\pos(\pi)-1$
and  $m(\pi)=m(\tau')+2i-1$.
Similarly, if $-n$ is an element of $T_i$ for some $i\in [k]$,
then removing it from $T_i$ yields a {\sc PSSP} $\tau'$ of $[n-1]$ into $k$ blocks
such that $\pos(\tau')=\pos(\pi)$ and $m(\pi)=m(\tau')+2i$.
If neither $n$ nor $-n$ is in any block of $\pi$, then $\pi \in B_{\subseteq}([n-1],k)$.

Thus
\begin{align*}
    S_B(n,k,q)
 &= q^{2k-1}(1+q)\sum\limits_{\pi \in B_{\subseteq}([n-1],k-1)} q^{m(\pi)}
    +(1+q)\sum_{i=1}^kq^{2i-1}\sum\limits_{\tau' \in B_{\subseteq}([n-1],k)}q^{m(\tau')}\\
 & \quad   +\sum\limits_{\pi \in B_{\subseteq}([n-1],k)} q^{m(\pi)}\\
 &= [2k+1]_qS_B(n-1,k,q)+q^{2k-1}(1+q)S_B(n-1,k-1,q).
\end{align*}
This finishes the proof.
\end{proof}

%Since the proof of \eqref{eq=CG} is the same as
%in \cite {CG07} using \eqref{def+m}, it is omitted.

%Note that the difference between the definitions of type B and D Stirling numbers of the second kind
%in terms of (ordered) signed partition.
\begin{defn}
We define the \emph{$q$-Stirling numbers of the second kind in type D} by
\begin{equation}\label{def+stirling+D+q}
S_D[n,k]:=\frac{1}{q^{k^2}[2]_q^{k}}\sum_{\pi \in D_{\subseteq}([n],k)}q^{m(\pi)}.
\end{equation}
\end{defn}

The following results are $q$-analogues of identities \eqref{rel+stirling+AB} and \eqref{rel+stirling+ABD},
which also show that $S_D[n,k]$ is a polynomial in $q$.
Let $S[n,k]_{q^2}$ denote $S[n,k]$ with $q$ replaced by $q^2$, i.e.,
\begin{equation*}
S[n,k]_{q^2}:=S[n,k]\, \big|_{q\leftarrow q^2}.
\end{equation*}

\begin{prop}\label{prop+rel+stirling+q+A+B+D}
Let $S_D[n,k]$ be defined by~\eqref{def+stirling+D+q}. Then the identities
\begin{align}
 S_B[n,k] &= \sum_{j=k}^n\binom{n}{j}[2]_{q}^{j-k}q^{j-k}S[j,k]_{q^2},\label{rel+stirling+AB+q}\\
 S_B[n,k] &= S_D[n,k]+n\cdot [2]_q^{n-k-1}q^{n-k-1}S[n-1,k]_{q^2}     \label{rel+stirling+ABD+q}
\end{align}
hold for $0 \leq k \leq n$.
\end{prop}

\begin{proof}
We first prove identity \eqref{rel+stirling+AB+q}.
Define the polynomial $\widetilde{B}_{n,k}(q)$ by
\begin{equation*}
  \widetilde{B}_{n,k}(q):=\sum_{\pi \in B([n],k)}q^{m(\pi)}.
\end{equation*}
Let $\pi$ and $\pi'$ be  {\sc SSPs} with $k$ blocks in two different nonempty subsets $\{i_1,i_2,\ldots,i_{\ell}\}$
and $\{j_1,j_2,\ldots,j_{\ell}\}$ of $[n]$, respectively.
Obviously, the set of all {\sc SSPs} $\pi$ of $\{i_1,i_2,\ldots,i_{\ell}\}$ and that of {\sc SSPs} $\pi'$ of $\{j_1,j_2,\ldots,j_{\ell}\}$ are equivalent regardless of the letters.
Then we can rewrite identity \eqref{eq=CG} as
\begin{equation*}
   q^{k^2}[2]_q^{k}S_B[n,k]
 = \sum_{j=k}^{n}\binom{n}{j}\widetilde{B}_{j,k}(q).
\end{equation*}
Thus, to prove identity \eqref{rel+stirling+AB+q}, it is sufficient to show that
\begin{equation*}
[2]_q^{n}q^{k(k-1)+n}S[n,k]_{q^2} = \widetilde{B}_{n,k}(q).
\end{equation*}

Next, we will prove that both sides of the above identity have the same recurrence relation and initial condition.
By the definition of {\sc SSP},
there exist two ways to get a {\sc SSP} of $[n]$ by inserting $n$ or $\overline{n}$ in one of $[n-1]$.

\begin{enumerate}[label=\roman*)]
\item [\rm (i)] The letter $n$ or $\overline{n}$ inserts a {\sc SSP} in $B([n-1],k-1)$
      and forms a new block listing the last position,
      which increases the statistic $m(\pi)$ by $2k-1$ and $2k$, respectively.
\item [\rm (ii)] The letter $n$ or $\overline{n}$ inserts the $i$th block of a {\sc SSP} in $B([n-1],k)$,
      which increases the statistic $m(\pi)$ by $2i-1$ and $2i$, respectively.
\end{enumerate}

From those, we have the recurrence relation
\begin{equation*}
   \widetilde{B}_{n,k}(q)
 = [2]_q\,q^{2k-1}\widetilde{B}_{n-1,k-1}(q)+q\cdot[2]_q\,[k]_{q^2}\widetilde{B}_{n-1,k}(q),
\end{equation*}
with the initial condition $\widetilde{B}_{0,0}(q)=1$.
Due to the recurrence relation~\eqref{rec+stirling+q+A} of $q$-Stirling numbers of the second kind in type A,
the desired result is obtained.

For identity \eqref{rel+stirling+ABD+q},
by identity \eqref{eq=CG} and the definition of $S_D[n,k]$, it suffices to show~that
\begin{equation*}
    n\cdot[2]_q^{n-1}q^{k(k-1)+n-1}S[n-1,k]_{q^2}
  = \sum_{i=1}^n\sum_{\pi \in B([n]\backslash \{i\},k)}q^{m(\pi)},
\end{equation*}
which is immediate by the above discussion.
\end{proof}

\begin{rem}
For nonnegative integers $n \geq k$ with $n \neq 1$, Bagno et al. \cite{BBG19} proved the following identity:
\begin{equation}\label{eq+eulerian+stir+D}
  S_D(n,k) = \frac{1}{2^kk!}\left[\sum_{\ell=0}^kD(n,\ell)\binom{n-\ell}{k-\ell}+n\cdot2^{n-1}(k-1)!S(n-1,k-1)\right],
\end{equation}
where $D(n,\ell)$ is the number of permutations in $\mathcal{D}_n$, which is the set of all signed permutations with even signs in $\B_n$, with $\ell$ descents,
see \cite[Section 11.5.4]{peter15} for more details.
As 
for the type D $q$-Stirling numbers of the second kind $S_D[n,k]$ defined by~\eqref{def+stirling+D+q}, we leave it as an open problem 
to  find  a $q$-analogue of identity~\eqref{eq+eulerian+stir+D} in the spirit of identities~\eqref{eq+eulerian+stir+A+q}  and~\eqref{eq+eulerian+stir+B+q+trans} for types A and B.

\end{rem}

\subsection{Falling factorials and $q$-Stirling numbers of the second kind in type D}
%\subsection{\texorpdfstring{$q-$}-falling factorials of type \texorpdfstring{D}- and \texorpdfstring{$q-$}-Stirling numbers of the second kind in type \texorpdfstring{D}-}

For the Stirling numbers of the second kind $S(n,k)$,
a well-known identity involving the connection between the standard basis of the polynomial ring $\R_n[t]$
and the basis consisting of falling factorials is that, for $n \in \N$ and $t \in \C$ , we have
\begin{equation}\label{eq+basis+falling+stir+A}
  t^n = \sum_{k=0}^{n} S(n,k)(t)_k,
\end{equation}
where $(t)_k=t(t-1)\cdots\left(t-(k-1)\right)$ and $(t)_0:=1$.

A classical combinatorial interpretation for~\eqref{eq+basis+falling+stir+A} pointed out that
$t^n$ is the number of all mappings from the set $[n]$ to the set $[t]$ ($t \in \N^{+}$)
and $S(n,k)(t)_k$ is the number of surjections that map the set $[n]$ to all $k$-subsets of $[t]$,
see \cite[Eq. (1.96)]{St97} for more details.
Similarly, for the Stirling numbers of the second kind in types B and D,
Bagno et al. \cite[Theorems 5.1 and 5.4]{BBG19} used a geometric method to obtain the following identities:

\begin{equation}\label{eq+basis+falling+stir+B}
  t^n = \sum_{k=0}^{n} S_B(n,k)(t)_k^B,
\end{equation}
where $(t)_k^B=(t-1)(t-3)\cdots\left(t-(2k-1)\right)$ and $(t)_0^B:=1$, and
\begin{equation}\label{eq+basis+falling+stir+D}
  t^n = \sum_{k=0}^{n} S_D(n,k)(t)_k^D + n\left((t-1)^{n-1} - (t)_{n-1}^D \right),
\end{equation}
where $(t)_k^D$ is defined as
\begin{equation*}
     (t)_{k}^D
 := \left\{
   \begin{array}{lll}
      1,                                      &  k=0;          \\
      (t-1)(t-3)\cdots(t-(2k-1)),             &  1 \leq k < n; \\
      (t-1)(t-3)\cdots(t-(2n-3))(t-(n-1)),  &  k=n.          \\
   \end{array}
   \right.
\end{equation*}

Naturally, those $q$-analogues for identities~\eqref{eq+basis+falling+stir+A} and~\eqref{eq+basis+falling+stir+B} were also given as
\begin{equation}\label{eq+basis+falling+stir+A+q}
  t^n = \sum_{k=0}^{n} S[n,k](t)_{k,q},
\end{equation}
where $(t)_{k,q}=t(t-[1]_q)\cdots\left(t-[k-1]_q\right)$ and $(t)_{0,q}:=1$ (see Carlitz~\cite[Eq. (3.1)]{Car48}),
and
\begin{equation}\label{eq+basis+falling+stir+B+q}
  t^n = \sum_{k=0}^{n} S_B[n,k](t)_{k,q}^B,
\end{equation}
where $(t)_{k,q}^B=(t-[1]_q)(t-[3]_q)\cdots\left(t-[2k-1]_q\right)$ and $(t)_{0,q}^B:=1$
(see Sagan and Swanson~\cite[Corallary 2.4]{SS22} and Komatsu et al.~\cite[Theorem 2.2]{KBG22}).

Define a \emph{$q$-falling factorial of type D} by
\begin{equation*}
     (t)_{k,q}^D
 := \left\{
   \begin{array}{lll}
      1,                                               &  k=0;          \\
      (t-[1]_q)(t-[3]_q)\cdots(t-[2k-1]_q),            &  1 \leq k < n; \\
      (t-[1]_q)(t-[3]_q)\cdots(t-[2n-3]_q)(t-[n-1]_q), &  k=n.          \\
   \end{array}
   \right.
\end{equation*}
We have a $q$-analogue of identity~\eqref{eq+basis+falling+stir+D} as follows.

\begin{prop}\label{prop+stir+D+q}
Let $S_D[n,k]$ be defined by~\eqref{def+stirling+D+q}. Then 
\begin{equation*}
  t^n = \sum_{k=0}^nS_D[n,k](t)_{k,q}^D+n(t-1)^{n-1}-[n]_{q}\,q^{n-1}(t)_{n-1,q}^D
\end{equation*}
for $n \in \N$ and $t \in \C$.
\end{prop}

\begin{proof}
From  equation~\eqref{rel+stirling+ABD+q} we derive the identity
\begin{equation}\label{rel+stirling+q+A+B+D}
     S_B[n,k]
  =  S_D[n,k]+n\cdot[2]_q^{n-k-1}q^{n-k-1}S[n-1,k]_{q^2}.
\end{equation}
Thus, multiplying  both sides of~\eqref{rel+stirling+q+A+B+D} by $(t)_{k,q}^D$  and  summing   over $0 \leq k \leq n$, we have
\begin{equation}\label{rel+stirling+q+A+B+D+GF}
     \sum_{k=0}^{n}S_B[n,k](t)_{k,q}^D
  =  \sum_{k=0}^{n}S_D[n,k](t)_{k,q}^D+\sum_{k=0}^{n-1}n\cdot[2]_q^{n-k-1}q^{n-k-1}S[n-1,k]_{q^2}(t)_{k,q}^D.
\end{equation}

First, for the left-hand side of~\eqref{rel+stirling+q+A+B+D+GF}, we have
\begin{align}\label{sum+stir+B+q}
      \sum_{k=0}^{n}S_B[n,k](t)_{k,q}^D
  &= \sum_{k=0}^{n-1}S_B[n,k](t)_{k,q}^D + S_B[n,n](t)_{n,q}^D-[n]_{q}\,q^{n-1}(t)_{n-1,q}^D + [n]_{q}\,q^{n-1}(t)_{n-1,q}^D \nonumber\\
  &= \sum_{k=0}^{n}S_B[n,k](t)_{k,q}^B + [n]_{q}\,q^{n-1}(t)_{n-1,q}^D                                                       \nonumber\\
  &= t^n + [n]_{q}\,q^{n-1}(t)_{n-1,q}^D,
\end{align}
where the second equality and last one use the facts $S_B[n,n]=1$ and
\begin{equation*}
(t)_{n,q}^B=(t)_{n,q}^D-[n]_{q}\,q^{n-1}(t)_{n-1,q}^D,
\end{equation*}
and identity~\eqref{eq+basis+falling+stir+B+q}, respectively.
In addition, for the second summation in the right-hand side of~\eqref{rel+stirling+q+A+B+D+GF}, we have
\begin{align}\label{sum+stir+D+q}
  &\quad \sum_{k=0}^{n-1}n\cdot[2]_q^{n-k-1}q^{n-k-1}S[n-1,k]_{q^2}(t)_{k,q}^D   \nonumber\\
  &=  n\cdot[2]_q^{n-1}q^{n-1}\sum_{k=0}^{n-1}S[n-1,k]_{q^2}
       \left(\frac{t-1}{[2]_{q}q}\right)\left(\frac{t-1}{[2]_{q}q}-[1]_{q^2}\right)\cdots
       \left(\frac{t-1}{[2]_{q}q}-[k-1]_{q^2}\right)                            \nonumber\\
  &=  n\cdot[2]_q^{n-1}q^{n-1}\left(\frac{t-1}{[2]_{q}q}\right)^{n-1}           \nonumber\\
  &=  n(t-1)^n,
\end{align}
where the second equality uses identity~\eqref{eq+basis+falling+stir+A+q}.
Combining~\eqref{rel+stirling+q+A+B+D+GF},~\eqref{sum+stir+B+q} and~\eqref{sum+stir+D+q}, we complete the proof.
\end{proof}

\section{Generalization to colored permutations}\label{section+algebraic+proof}

In this section, instead of proving Theorem \ref{thm+stirling+frobenius+B} by an algebraic proof,
we shall prove a more general identity.
Define the \emph{$r$-colored $q$-Stirling numbers of the second kind $S_{r}[n,k]$} by the recurrence relation
\begin{equation}\label{def+stirling+r}
  S_{r}[n,k]:= S_{r}[n-1,k-1] + [rk+1]_q\,S_{r}[n-1,k]
\end{equation}
with the initial conditions $S_{r}[0,k]=\delta_{0k}$.
%In fact, the numbers $S_{r}[n,k]$ can be looked as the $q$-analogue of $r$-Whitney numbers (see)

It is not difficult to verify (see \cite[Theorem 1]{RW04} for a more general result) that
\begin{equation}\label{eq+basis+falling+stir+r+q}
    t^{n}
  = \sum_{k=0}^{n}S_r[n,k](t)_{k,q}^r,
\end{equation}
where $(t)_{k,q}^r=(t-[1]_q)(t-[r+1]_q)\cdots\left(t-[r(k-1)+1]_q\right)$ and $(t)_{0,q}^r:=1$.
Using Rook theory, Remmel and Wachs gave a combinatorial interpretation of
identity \eqref{eq+basis+falling+stir+r+q} in \cite[Theorem 7]{RW04}.
%which was also a special case in \cite[Eqn. (34)]{RW04} by a direct definition.

Substituting $t$ by $[rm+1]_{q}$ in \eqref{eq+basis+falling+stir+r+q} yields
\begin{equation*}
    [rm+1]_{q}^{n}
  = \sum_{k=0}^{n}q^{r\binom{k+1}{2}+(1-r)k}[r]_q^k\,[k]_{q^r}!S_r[n,k]
    {m \brack k}_{q^r},
\end{equation*}
which, by \eqref{q2}, is equivalent to the generating function identity,
%We derive straightforwardly from \eqref{def+stirling+r} that
%\begin{align}
%\sum_{n\geq k} S_r[n,k]u^{n-k}=
%\frac{1}{\prod_{l=1}^k(1-u[rl+1]_q)}.
%\end{align}
\begin{equation}\label{eq+eulerian+sturling+color+q}
     \sum_{k=0}^n\frac{q^{r\binom{k+1}{2}+(1-r)k}[r]_q^k\,[k]_{q^r}!S_{r}[n,k]\,t^k}{\prod_{i=0}^k(1-tq^{ri})}
   = \sum_{m=0}^{\infty}[rm+1]_q^{n}\,t^{m}.
\end{equation}

%In particular, identity~\eqref{eq+basis+falling+stir+r+q} implies~\eqref{eq+basis+falling+stir+A+q}
%and~\eqref{eq+basis+falling+stir+B+q} for $r=1$ and $r=2$, respectively.
The colored permutations group of $n$ letters with $r$ colors can be looked as the wreath product group
\begin{equation*}
\Z_r \wr \S_n = \Z^{r} \times \S_n,
\end{equation*}
which consists of all permutations $\pi \in [0,r-1]\times [n]$. Namely, the element in
$\Z_r \wr \S_n$ is thought of as $\pi = \pi_1^{z_1}\pi_2^{z_2} \cdots \pi_n^{z_n}$,
where $z_i \in [0,r-1]$ and $\pi_1\pi_2 \cdots \pi_n \in \S_n$.
Define the following total order relation on the elements of $\Z_r \wr \S_n$:
\begin{equation*}
n^{r-1}<\cdots<n^1<\cdots<1^{r-1}<\cdots<1^{1}<0<1<\cdots<n,
\end{equation*}
where $k^{0}$ is replaced with $k$ for $k \in [n]$.

An integer $i \in \{0\} \cup [n-1]$ is called a \emph{descent} of $\pi \in  \Z_r \wr \S_n$
if $\pi_i^{z_i} > \pi_{i+1}^{z_{i+1}}$, where $\pi_0^{z_0}=0$.
Let  $\Des_r(\pi)$ denote the descent set of $\pi \in  \Z_r \wr \S_n$ and  $\des_r(\pi)$ the number of descents of $\pi$, i.e., $|\Des_r(\pi)|$.
The $r$-colored Eulerian number $A_{n,k}^r$ is the number of all colored permutations in $\Z_r \wr \S_n$ with $k$ descents.
For each $\pi \in \Z_r \wr \S_n$, as in \cite{AR01}, define the \emph{$r$-flag-major index} of $\pi$ by
\begin{equation}
  \fmaj_r(\pi) := r\sum_{i \in \Des_r(\pi)}i+\sum_{i=1}^nz_i.
\end{equation}
A $q$-analogue of the $r$-colored Eulerian polynomial $A_n^r(t,q)$ is defined by
\begin{equation}\label{def+eulerian+color+q}
  A_n^r(t,q) : = \sum_{\pi \in \Z_r \wr \S_n}t^{\des_r(\pi)}q^{\fmaj_r(\pi)}
  =\sum_{k=0}^nA_{n,k}^r(q)t^k.
\end{equation}
When $r$ takes $1$ and $2$, \eqref{def+eulerian+color+q} reduces to
\eqref{def+q+eulerian+poly+A} and \eqref{def+q+eulerian+poly+B}, respectively.
The following Carlitz's identity for $\Z_r \wr \S_n$ was proved in \cite[Proposition 8.1]{BZ11} and \cite[Theorem 9]{CM11}
\begin{equation}\label{carlitz-color}
     \frac{A_n^{r}(t,q)}{\prod_{i=0}^n(1-tq^{ri})}
   = \sum_{m=0}^{\infty}[rm+1]_q^{n}\,t^{m}.
\end{equation}

Combining \eqref{eq+eulerian+sturling+color+q} and \eqref{carlitz-color} we obtain the following identity.

\begin{prop}\label{prop+eulerian+stirling+fro}
For the polynomials $S_r[n,k]$ in~\eqref{def+stirling+r} and $A_n^{r}(t,q)$ in~\eqref{def+eulerian+color+q},
the $q$-Frobenius formula holds
\begin{equation*}
    \frac{A_n^{r}(t,q)}{\prod_{i=0}^n(1-tq^{ri})}
  = \sum_{k=0}^n\frac{q^{r\binom{k+1}{2}+(1-r)k}[r]_q^k\,[k]_{q^r}!S_{r}[n,k]t^k}{\prod_{i=0}^k(1-tq^{ri})}.
\end{equation*}
\end{prop}

The following result is  a $q$-analogue of Theorem 6.6 in \cite{BBG19}
about an identity between the $r$-colored Stirling numbers of the second kind $S_r(n,k)$
(the sequence defined by~\eqref{def+stirling+r} when $q=1$, see also\cite[Section 6.1]{BBG19})
and $r$-colored Eulerian numbers $A_{n,k}^r$.

\begin{thm}\label{thm+stirling+frobenius+r}
For the $r$-colored $q$-Stirling numbers of the second kind $S_{r}[n,k]$ in~\eqref{def+stirling+r}
and $q$-Eulerian numbers $A_{n,k}^r(q)$ in~\eqref{def+eulerian+color+q}, we have the identity
\begin{equation}\label{r-stirling-eulerian}
    q^{r\binom{k+1}{2}+(1-r)k}[r]_q^k\,[k]_{q^r}!S_{r}[n,k]
  = \sum_{\ell=0}^{k}q^{rk(k-\ell)}A_{n,\ell}^{r}(q)
    {n-\ell \brack k-\ell}_{q^r}
\end{equation}
for $0 \leq k \leq n$.
\end{thm}

\begin{proof}
Summing for both sides of \eqref{r-stirling-eulerian} multiplying by $t^k/\prod_{i=0}^k(1-tq^{ri})$ over all $k$,
it is clear that \eqref{r-stirling-eulerian} is equivalent to
\begin{equation*}
    \sum_{k=0}^n\frac{q^{r\binom{k+1}{2}+(1-r)k}[r]_q^k\,[k]_{q^r}!S_{r}[n,k]t^k}{\prod_{i=0}^k(1-tq^{ri})}
  = \sum_{k=0}^n\sum_{\ell=0}^{k}\frac{q^{rk(k-\ell)}A_{n,\ell}^{r}(q)t^k}{\prod_{i=0}^k(1-tq^{ri})}
    {n-\ell \brack k-\ell}_{q^r}.
\end{equation*}
By Proposition \ref{prop+eulerian+stirling+fro}, it is sufficient to show that
\begin{equation*}
    \frac{A_n^{r}(t,q)}{\prod_{i=0}^n(1-tq^{ri})}
  = \sum_{k=0}^n\sum_{\ell=0}^{k}\frac{q^{rk(k-\ell)}A_{n,\ell}^{r}(q)t^k}{\prod_{i=0}^k(1-tq^{ri})}
    {n-\ell \brack k-\ell}_{q^r},
\end{equation*}
or equivalently,
\begin{equation*}
    \frac{A_n^{r}(t,q)}{\prod_{i=0}^n(1-tq^{ri})}
  = \sum_{\ell=0}^{n}A_{n,\ell}^{r}(q)t^{\ell}
    \sum_{k=\ell}^n\frac{(tq^{r \cdot k})^{k-\ell}}{\prod_{i=0}^k(1-tq^{ri})}
    {n-\ell \brack k-\ell}_{q^r},
\end{equation*}
which will follow from the following identity
\begin{equation}\label{eq+q+extension}
    \frac{1}{\prod_{i=0}^n(1-tq^{ri})}
  = \sum_{k=\ell}^n\frac{(tq^{r \cdot k})^{k-\ell}}{\prod_{i=0}^k(1-tq^{ri})}{n-\ell \brack k-\ell}_{q^r}
\end{equation}
for $0 \leq \ell \leq n$.
That is to say, the index $\ell$ does not affect the summation in the right-hand side of~\eqref{eq+q+extension}.
Substituting $q^r\to q$ and applying \eqref{q2} to extract the coefficients of $t^m$
on both sides of~\eqref{eq+q+extension} we obtain
\begin{align*}
     {n+m \brack m}_{q}
  &= \sum_{k=\ell}^n{n-\ell \brack k-\ell}_{q}
     {k+m-(k-\ell) \brack m-(k-\ell)}_{q}q^{ k (k-\ell)}\\
%\end{equation*}
%for $0 \leq \ell \leq k$. this can be written as
%\begin{equation*}
%    {n+m \brack  m}_{q}
  &= \sum_{k=0}^{n-\ell}{n-\ell\brack k}_{q}{m+\ell \brack m-k}_{q}q^{k(k+\ell)},
\end{align*}
which is a $q$-analogue of Chu-Vandermonde summation \cite[Eq. (3.3.10)]{An76}.
\end{proof}

Following the recurrence \eqref{def+stirling+r}, we have
$S_{1}[n,k]=S[n+1,k+1]$ and $S_{2}[n,k]=S_B[n,k]$.
When $r=1$ and $r=2$, identity \eqref{r-stirling-eulerian} (Theorem \ref{thm+stirling+frobenius+r}) reduces to
\eqref{eq+eulerian+stir+A+q} and \eqref{eq+eulerian+stir+B+q+trans}, respectively.
Indeed, the case $r=2$ is obvious, i.e., Theorem \ref{thm+stirling+frobenius+B} is a special case of Theorem \ref{thm+stirling+frobenius+r}.
For $r=1$, Theorem \ref{thm+stirling+frobenius+r} reduces to
\begin{equation}\label{eq+euler+stirling+A}
    q^{\binom{k+1}{2}}[k]_{q}!\,S[n+1,k+1]
  = \sum_{\ell=0}^{k}q^{k(k-\ell)}A_{n,\ell}(q)
    {n-\ell \brack k-\ell}_{q},
\end{equation}
which is equivalent to identity~\eqref{eq+eulerian+stir+A+q}.
By \eqref{eq+eulerian+stir+A+q}, the right-hand side of \eqref{eq+euler+stirling+A} equals
\begin{align*}
  &\quad  \sum_{\ell=0}^{k}q^{k(k-\ell)}A_{n,\ell}(q){n-\ell \brack k-\ell}_{q}                   \\
  &= \sum_{\ell=1}^{k}q^{k(k-(\ell-1))}A_{n,\ell-1}(q)
      {n-(\ell-1) \brack k-(\ell-1)}_{q}+A_{n,k}(q)                                               \\
  &= \sum_{\ell=1}^{k}q^{k(k+1-\ell)}A_{n,\ell-1}(q)\left(q^{k+1-\ell}{n-\ell \brack k+1-\ell}_{q}
     + {n-\ell \brack k-\ell}_{q}\right)+A_{n,k}(q)                                               \\
  &= \sum_{\ell=1}^{k+1}q^{(k+1)(k+1-\ell)}A_{n,\ell-1}(q){n-\ell \brack k+1-\ell}_{q}
     + q^k\sum_{\ell=1}^{k}q^{k(k-\ell)}A_{n,\ell-1}(q){n-\ell \brack k-\ell}_{q}                 \\
  &= q^{\binom{k+1}{2}}[k+1]_{q}!S[n,k+1]+ q^kq^{\binom{k}{2}}[k]_{q}!S[n,k],
\end{align*}
which yields~\eqref{eq+euler+stirling+A} by recurrence relation~\eqref{rec+stirling+q+A} of $S[n,k]$.
Inversely, starting from~\eqref{eq+euler+stirling+A},
the above last equality shows that~\eqref{eq+eulerian+stir+A+q}
follows from~\eqref{eq+euler+stirling+A} by induction on $k$ for fixed $n$.

In addition, by \eqref{eq+eulerian+stir+A+q} and \eqref{eq+q+extension}, we have
the following $q$-Frobenius formula \cite[Eq. (4.1)]{Gar79} related to $q$-Stirling numbers of the second kind
and $q$-Eulerian polynomials of type~A:
\begin{equation}\label{eq+eulerian+stir+A+q+frac}
    \frac{tA_n(t,q)}{\prod_{i=0}^n(1-tq^{i})}
  = \sum_{k=0}^n\frac{q^{\binom{k}{2}}[k]_{q}!S[n,k]t^k}{\prod_{i=0}^k(1-tq^{i})}.
\end{equation}
Following the above discussion, it is clear that identity \eqref{eq+eulerian+stir+A+q+frac} is a special case of
Proposition \ref{prop+eulerian+stirling+fro} for $r=1$.

\begin{rem}
Similar to the combinatorial proofs of \eqref{eq+eulerian+stir+A+q} and \eqref{eq+eulerian+stir+B+q+trans}, 
it is natural to  ask for a combinatorial proof of identity \eqref{r-stirling-eulerian}.
One difficulty for such a proof is that 
a counterpart  of the $q$-symmetry~\eqref{q-symmetry}  is missing  for $A_{n,k}^r(q)$. Note that the $r$-colored Eulerian polynomials 
$A_{n}^r(t,1)$ are not symmetric for $r\geq3$.
We leave it as an open problem to give a combinatorial proof of identity \eqref{r-stirling-eulerian}.
\end{rem}

\section*{Acknowledgements}
The first author was supported by the \emph{China Scholarship Council}.
This work was done during his visit  at  Universit\'e Claude Bernard Lyon 1 in 2022-2023.

\end{document}